\documentclass[11pt]{amsart}

\usepackage[article]{robertpreamble}
\usepackage{geometry}
\begin{document}


\title{Nested nodal loops for sums of Laplace eigenfunctions}
\author{Robert Koirala}
\address{Department of Mathematics, University of California San Diego}
\email{rkoirala@ucsd.edu}
\date{\today}

\begin{abstract}
    We study nested loops in zero sets of sums of Laplace eigenfunctions on closed surfaces. In the real-analytic category, answering a question of Logunov, we prove a uniform bound for the number of rooted double nests in terms of the surface, the root, and the spectral cutoff. We show that this analyticity hypothesis is sharp: on a smooth sphere, a linear combination of eigenfunctions with eigenvalues \(0\) and \(2\) can have infinitely many rooted double nests. We also answer a question of Logunov and Nadirashvili by constructing a planar biharmonic function whose nodal set contains a double nest, and we prove a quantitative bound for entire biharmonic functions of polynomial growth. The biharmonic construction gives a nodal-set manifestation of the failure of the Boggio--Hadamard conjecture from the 1900s.
\end{abstract}

\maketitle 


\section{Introduction}

Let \((M^2,g)\) be a closed Riemannian surface, and let \(\phi,\psi\) be two Laplace eigenfunctions:
\begin{align}
    -\Delta_g\phi=\lambda\phi,\qquad -\Delta_g\psi=\mu\psi .
\end{align}
We study the topology of the nodal set \(u^{-1}(0)\) of a nontrivial linear combination \(u=a\phi+b\psi\).

\begin{question}[Logunov]\label{question:surface}
    Can the nodal set \(u^{-1}(0)\) of a linear combination \(u=a\phi+b\psi\) contain arbitrarily many pairs of disjoint nested nodal loops?
\end{question}

Here a nodal loop means a connected component of \(u^{-1}(0)\) homeomorphic to \(S^1\). Informally, two loops are nested if one lies inside the disk bounded by the other, as in
\begin{tikzpicture}[baseline=-0.5ex,scale=0.7]
    \draw (0,0) circle (0.25);
    \draw (0,0) circle (0.10);
\end{tikzpicture}.
See Definitions~\ref{definition:nested-loops} and \ref{definition:nested-loops-surface}.

The motivation for Question~\ref{question:surface} is to understand which features of classical nodal geometry survive after one passes from a single eigenfunction to a finite spectral sum. For a single eigenfunction, the nodal set is subject to strong topological and geometric constraints. Courant's nodal domain theorem bounds the number of nodal domains of a \(k\)-th eigenfunction by \(k\) \cite{MR65391}; see also \cite{MR4704545} for a local version of Courant's theorem. Cheng's theorem shows that, on a closed surface, the nodal set of an eigenfunction has the structure of a finite topological graph \cite{MR397805}. Further finiteness results for level and nodal sets of eigenfunctions, including bounds for their Betti numbers, have been studied by Lin--Liu \cite{lin-liu2016-betti-numbers-of-level-sets}. See also \cite{hallgren2025structure} and the references therein for recent geometric and measure-theoretic developments.

The situation changes substantially for linear combinations of eigenfunctions. Some analytic estimates, especially quantitative unique continuation and volume estimates for nodal sets, extend to finite spectral sums \cite{MR1205487}. The corresponding topological theory, however, is much more fragile. The extended Courant property, which would bound the number of nodal domains of a linear combination of the first \(n\) eigenfunctions by \(n\), is false in general. Even more strikingly, B\'erard--Charron--Helffer and Buhovsky--Logunov--Sodin constructed smooth closed surfaces and eigenfunctions \(\Phi\) for which the nodal sets of \(\Phi-1\) have infinitely many connected components \cite{MR4467996,MR3890963,MR4060310,MR4082305,MR4281621,MR4190398}. Thus, when \(\lambda\neq\mu\), the function \(u=a\phi+b\psi\) need not inherit the topological restrictions available for a single eigenfunction.

Question~\ref{question:surface} asks whether a more rigid topological feature might nevertheless remain controlled. Rather than counting all components, it counts pairs of nested nodal loops. One source of hope for such rigidity is that \(u\) is annihilated by the fourth-order operator \((-\Delta_g-\lambda)(-\Delta_g-\mu).\) Thus, if one rescales \(u\) near a point \(p\) by setting \(U_\rho(y)=u(\exp_p(\rho y)),\) then, in the limit \(\rho\to0\), the leading local model is the planar biharmonic equation
\begin{align}
    \Delta\Delta U=0.
\end{align}
Planar biharmonic functions, which are more accessible than global spectral sums, therefore provide a natural local model for small-scale nodal configurations of sums of two eigenfunctions. This suggests the following Euclidean model problem.

\begin{question}[Logunov and Nadirashvili]\label{question:Euclidean-counter-example}
    Let \(u\) solve \(\Delta\Delta u=0\) on \(\mathbb R^2\). Can the nodal set
    \(u^{-1}(0)\) have a pair of nested loops?
\end{question}
This question is also related to the classical Boggio--Hadamard conjecture on positivity of Green's functions for the biharmonic operator, which would imply a maximum principle \cite{Boggio1901, Boggio1905,Hadamard1908, MR1696190}. The conjecture is guided by the intuition that a clamped plate pushed downward should bend downward everywhere. The connection with nested ovals is as follows. Suppose that a biharmonic function has a regular isolated outer nodal oval \(\gamma\), and let \(\Omega\) be the disk bounded by \(\gamma\). After changing sign if necessary, the function is negative in the inward collar of \(\partial\Omega\). Hence
\begin{align}
    \Delta^2 u=0 \quad \text{in }\Omega,\qquad u=0,
    \qquad \partial_\nu u\le0 \quad \text{on }\partial\Omega,
\end{align}
where \(\nu\) denotes the inward normal. If a Boggio--Hadamard-type maximum principle held in this setting, it would force \(u<0\) in \(\Omega\) and would therefore exclude an inner nodal oval. The same argument rules out double nests for planar harmonic functions. However, the Boggio--Hadamard conjecture was disproved by Duffin and Garabedian, with an elementary elliptic counterexample later given by Shapiro--Tegmark \cite{MR29021, MR46440, MR1267051}. See also \cite{MR2591975,MR1696190}. Since the biharmonic maximum principle fails in general, it is not a priori clear whether biharmonicity alone forbids a double nest.
\begin{figure}[!httb]
    \centering
    \includegraphics[width=0.75\linewidth]{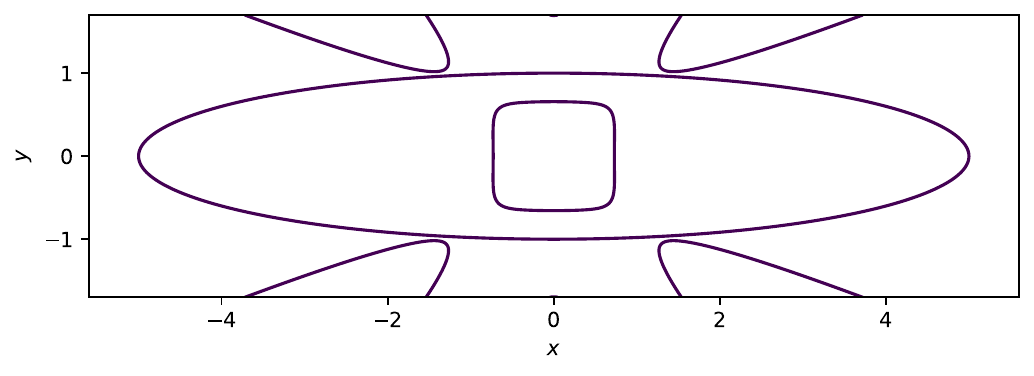}
    \caption{Nodal set of the biharmonic polynomial \(u(x,y)=Q(x,y)P(x,y)\), where \(P\) and \(Q\) are defined in \eqref{eq:biharmonic-counterexample}. The nodal set contains an outer ellipse and an inner closed loop, giving a double nest.}
    \label{fig:biharmonic-counterexample}
\end{figure}

The model problem is already delicate. If a biharmonic function \(u\) vanishes on a circle \(\partial B_R\), then Almansi's decomposition \(u(x,y)=h_0(x,y)+(x^2+y^2)h_1(x,y)\) and the maximum principle for harmonic functions imply
\begin{align*}
    u=((x^2+y^2)-R^2)h_1 \quad\text{in }B_R,
\end{align*}
with \(h_0,h_1\) harmonic. Since \(x^2+y^2<R^2\) in \(B_R\), any inner nodal oval would have to come from the zero set of \(h_1\). This is impossible for nontrivial harmonic \(h_1\), by the maximum principle. Thus a circular outer nodal loop cannot enclose another nodal oval. An ellipse, however, breaks this rigidity. A straightforward computation shows that \(u=QP\) is biharmonic, where
\begin{align}\label{eq:biharmonic-counterexample}
    \begin{split}
        Q(x,y)&=25-x^2-25y^2,\\
        P(x,y)&=8254531x^6-54879201x^4y^2+24763593x^4 +36742485x^2y^4\\
        &\qquad-12556530x^2y^2-1856167y^6-556035y^4 +19758600y^2-8254531.
    \end{split}
\end{align}
The nodal set of this polynomial contains a double nest, as illustrated in Figure~\ref{fig:biharmonic-counterexample}. Thus the simplest Euclidean biharmonic model does allow nested nodal loops. Moreover, the construction can be replicated. Given pairwise disjoint closed disks \(D_1,\dots,D_N\subset B(0,R)\), define a biharmonic function \(v\) near their union by placing translated and rescaled copies of \(u\) on the disks \(D_i\). The Lax--Malgrange--Runge approximation theorem then gives an entire biharmonic function \(u_N\) that approximates \(v\) sufficiently well on these disks. Since the relevant nodal loops of \(u\) are regular, \(u_N\) has at least \(N\) double nests. Thus entire biharmonic functions can have arbitrarily many double nests if no growth condition is imposed.

The main results below show that two kinds of restrictions recover finiteness. First, polynomial growth bounds the number of double nests for entire biharmonic functions. Second, real analyticity bounds the number of rooted double nests for finite spectral sums on closed surfaces. The analytic hypothesis is sharp: in the smooth category, even a shift of a single eigenfunction on \(S^2\) can have infinitely many rooted double nests.

We begin by fixing the notions of nesting used in the statements of the main results.

\begin{definition}\label{definition:nested-loops}
    Let \(u\) be a continuous function on \(\R^2\). A \emph{nodal oval} of \(u\) is a connected component of \(u^{-1}(0)\) that is homeomorphic to \(S^1\). If \(\gamma\) is such an oval, let \(D(\gamma)\) denote the bounded component of \(\R^2\setminus\gamma\). A \emph{double nest} is an ordered pair \((\gamma_{\rm in},\gamma_{\rm out})\) of distinct nodal ovals such that \(\overline{D(\gamma_{\rm in})}\subset D(\gamma_{\rm out})\). We denote the number of such pairs by \(N_{\mathrm{dn}}(u)\).
\end{definition}

\begin{definition}\label{definition:nested-loops-surface}
    Let \(M\) be a closed surface, let \(p\in M\), and let \(u\) be a continuous function on \(M\). A nodal oval \(\gamma\subset u^{-1}(0)\) is called a \emph{\(p\)-oval} if \(\gamma\subset M\setminus\{p\}\) and one component of \(M\setminus\gamma\) is an open disk whose closure does not contain \(p\). We denote this disk by \(D_p(\gamma)\) and call it the \emph{inside} of \(\gamma\) relative to \(p\). Two \(p\)-ovals form a \emph{\(p\)-rooted double nest} if \(\overline{D_p(\gamma_{\rm in})}\subset D_p(\gamma_{\rm out})\). We denote the number of \(p\)-rooted double nests by \(N_{\mathrm{dn}}^p(u)\).
\end{definition}

\begin{remark}
    The root is essential on a closed surface. On \(S^2\), every Jordan curve bounds two disks, so without a fixed root any two disjoint ovals can be declared nested by choosing the side of the outer oval after the fact. Fixing \(p\) is the closed-surface analogue of fixing the point at infinity in \(\R^2\simeq S^2\setminus\{p\}\).
\end{remark}

Although the example in \eqref{eq:biharmonic-counterexample} shows that double nests occur in the Euclidean biharmonic setting, our first result proves that polynomial growth prevents arbitrarily many of them.

\begin{theorem}\label{thm:plane-growth}
    Let \(u\not\equiv0\) be an entire biharmonic function on \(\R^2\), that is, \(\Delta\Delta u=0.\) Assume that, for some \(d\ge0\) and \(C<\infty\),
    \begin{align}
        |u(x,y)|\le C(1+x^2+y^2)^{d/2}.
    \end{align}
    Then, with \(m\coloneqq \floor{d}\), the number of double nests of the nodal set of \(u\) is bounded by
    \begin{align}
        N_{\mathrm{dn}}(u)\le \binom{\frac{m(m-1)}{2}}{2}.
    \end{align}
\end{theorem}

Our second result proves uniform boundedness of rooted double nests in the real-analytic category.

\begin{theorem}\label{thm:analytic}
    Let \((M^2,g)\) be a closed real-analytic Riemannian surface, let \(p\in M\), and let \(\Lambda<\infty\). Then there exists a constant \(C=C(M,g,p,\Lambda)<\infty\) such that every nonzero \(u\in \bigoplus_{\lambda_j\le\Lambda} \ker(-\Delta_g-\lambda_j)\) satisfies
    \begin{align}
        N_{\mathrm{dn}}^p(u)\le C.
    \end{align}
    In particular, every nonzero linear combination of two eigenfunctions with eigenvalues at most \(\Lambda\) has uniformly bounded \(p\)-rooted double nests.
\end{theorem}

The final result shows that the real-analyticity assumption in Theorem~\ref{thm:analytic} is sharp. In the smooth category, the conclusion fails on \(S^2\), and even for linear combinations of eigenfunctions with eigenvalues at most \(2\). However, the construction depends on the topology of \(S^2\) and does not seem to extend to other surfaces.

\begin{theorem}\label{thm:smooth-counterexample}
    Fix any point \(p\in S^2\). There exist a smooth Riemannian metric \(g\) on \(S^2\) and a Laplace eigenfunction \(\Phi\) of \((S^2,g)\), satisfying \(-\Delta_g\Phi=2\Phi,\) such that the nodal set of \(u\coloneqq \Phi-1\) contains infinitely many \(p\)-rooted double nests.
\end{theorem}

\subsection*{Organization}

In Section~\ref{sec:euclidean-toy-model}, we prove Theorem~\ref{thm:plane-growth}. In Section~\ref{sec:analytic}, we prove Theorem~\ref{thm:analytic}. Finally, in Section~\ref{sec:counter-example}, we prove Theorem~\ref{thm:smooth-counterexample}.

\subsection*{Acknowledgment}

We are grateful to Bennett Chow for his support and encouragement. We thank Alexander Logunov for bringing the historical context of Question~\ref{question:Euclidean-counter-example} to our attention, and Bogdan Georgiev and Javier Gómez-Serrano for their interest in this work.

\section{Bounded number of nested loops in the Euclidean setting}\label{sec:euclidean-toy-model}

\begin{proof}[Proof of Theorem~\ref{thm:plane-growth}]
    Let \(m\coloneqq \lfloor d\rfloor\). We first use the standard Liouville argument for entire biharmonic functions with polynomial growth. By interior estimates for \(\Delta^2\), for every multi-index \(\alpha\) and every \(z\in\R^2\),
    \begin{align}
        |D^\alpha u(z)| \le C_\alpha R^{-|\alpha|} \sup_{B_R(z)} |u|.
    \end{align}
    Using the polynomial growth assumption and letting \(R\to\infty\), we find that \(D^\alpha u(z)=0\) whenever \(|\alpha|>d\). Hence all derivatives of order \(m+1\) vanish identically, so \(u\) is a polynomial of degree at most \(m\).

    It remains to bound the number of oval components of the real algebraic curve \(u=0\). Let \(f\) be the square-free part of \(u\). Then \(f\in\mathbb R[x,y]\), \(f^{-1}(0)=u^{-1}(0)\), and \(n\coloneqq \deg f\le \deg u\le m.\) Thus the nodal ovals of \(u\) are precisely the oval components of \(f^{-1}(0)\).

    If \(n\le1\), then \(f^{-1}(0)\) has no oval components, and there is nothing to prove. Assume therefore that \(n\ge2\). Since rotations preserve oval components and nesting relations, we may rotate the coordinates. We choose the rotation generically so that no irreducible component of the complex curve \(f=0\) is a vertical line. Equivalently, \(\gcd_{\mathbb C[x,y]}(f,f_y)=1\). Indeed, write \(f=\prod_i g_i\) as a product of distinct irreducible factors over \(\mathbb C\). If some \(g_i\) divides \(f_y\), then reducing the identity for \(f_y\) modulo \(g_i\) gives \(g_i\mid (g_i)_y.\) Since \(\deg (g_i)_y<\deg g_i\) unless \((g_i)_y\equiv0\), this can happen only when \(g_i\in\mathbb C[x]\). As \(g_i\) is irreducible over \(\mathbb C\), this means that \(g_i\) is a vertical complex line. A generic rotation avoids this possibility.

    By B\'ezout's theorem, the number of complex solutions of \(f(x,y)=0, f_y(x,y)=0,\) counted with multiplicity, is at most \(\deg(f)\deg(f_y)\le n(n-1)\le m(m-1).\) In particular,
    \begin{align}
        \#\{(x,y)\in\R^2:f(x,y)=0,\ f_y(x,y)=0\}
        \le m(m-1).
    \end{align}

    Now let \(\gamma\) be an oval component of \(f^{-1}(0)\). The function \(x|_\gamma\) attains a maximum and a minimum on \(\gamma\). These occur at two distinct points, since otherwise \(x\) would be constant on \(\gamma\), forcing \(\gamma\) to lie in a vertical line. At each of these two points, either the curve \(f=0\) is singular, in which case \(f_y=0\), or the curve is smooth, in which case the tangent line is vertical, and again \(f_y=0\). Therefore each oval component contributes at least two distinct real solutions of the system \(f=0, f_y=0.\) Distinct oval components contribute disjoint sets of such points. Hence the number of oval components is at most
    \begin{align}
        \frac12\#\{f=0,\ f_y=0\} \le \frac{m(m-1)}2.
    \end{align}
    Finally, every double nest is determined by a pair of distinct ovals. For two distinct ovals, at most one ordering can satisfy \(\overline{D(\gamma_{\rm in})}\subset D(\gamma_{\rm out}).\) Therefore
    \begin{align}
        N_{\mathrm{dn}}(u) \le \binom{L_m}{2}, \qquad L_m\coloneqq \frac{m(m-1)}2.
    \end{align}
    This is the claimed bound.
\end{proof}

\section{Bounded number of nested loops in the real-analytic category}\label{sec:analytic}

\begin{proof}[Proof of Theorem~\ref{thm:analytic}]
    Let
    \begin{align}
        E_\Lambda = \bigoplus_{\lambda_j\le \Lambda} \ker(-\Delta_g-\lambda_j).
    \end{align}
    If \(E_\Lambda=\{0\}\), there is nothing to prove. Otherwise, let \(N=\dim E_\Lambda\) and choose a basis \(e_1,\dots,e_N\) of \(E_\Lambda\). Since \(g\) is real analytic, analytic elliptic regularity implies that each \(e_j\) is real analytic. Because multiplying by a nonzero scalar does not change the nodal set, it is enough to consider normalized sums
    \begin{align}
        u_a(x)=\sum_{j=1}^N a_j e_j(x),\qquad a=(a_1,\dots,a_N)\in S^{N-1}.
    \end{align}
    Define \(F:M\times S^{N-1}\to\mathbb R\) by \(F(x,a)=u_a(x)\). Then \(F\) is real analytic, and hence \(\mathcal Z\coloneqq F^{-1}(0)\) is a compact subanalytic subset of \(M\times S^{N-1}\). Let \(\pi:\mathcal Z\to S^{N-1}\) be the restriction of the projection onto the second factor. This is a proper subanalytic map, and its fiber over \(a\) is naturally homeomorphic to \(u_a^{-1}(0)\). By Hardt triviality in the subanalytic category \cite{MR564475}, there is a finite partition
    \begin{align}
        S^{N-1}=A_1\sqcup\cdots\sqcup A_K
    \end{align}
    into subanalytic pieces such that \(\pi\) is topologically trivial over each \(A_k\). Consequently, only finitely many homeomorphism types occur among the compact fibers \(u_a^{-1}(0)\).

    For each \(a\), the fiber \(u_a^{-1}(0)\) is a compact subanalytic subset of the real-analytic surface \(M\). By the triangulation theorem for subanalytic sets \cite{MR454051}, it is homeomorphic to a finite simplicial complex. In particular, it has finitely many connected components. Since only finitely many fiber homeomorphism types occur as \(a\) varies, there exists a constant \(L_0<\infty\) such that every fiber \(u_a^{-1}(0)\) has at most \(L_0\) connected components.

    Every \(p\)-oval is, by definition, a connected component of \(u_a^{-1}(0)\). Hence \(u_a^{-1}(0)\) has at most \(L_0\) \(p\)-ovals. For two distinct \(p\)-ovals, at most one ordering can satisfy \(\overline{D_p(\gamma_{\rm in})} \subset D_p(\gamma_{\rm out}).\) Therefore
    \begin{align}
        N_{\mathrm{dn}}^p(u_a) \le \binom{L_0}{2}
    \end{align}
    for every \(a\in S^{N-1}\).

    Finally, every nonzero element of \(E_\Lambda\) is a nonzero scalar multiple of some \(u_a\). Thus the same estimate holds for every nonzero \(u\in E_\Lambda\). Taking \(C(M,g,p,\Lambda)=\binom{L_0}{2}\) proves the desired uniform bound.
\end{proof}

\begin{remark}
    The proof gives more than an oval-component bound. Since the fibers form a subanalytic family with only finitely many homeomorphism types, any topological count determined by the homeomorphism type of the nodal set is uniformly bounded. In particular, on the one-dimensional part of the nodal set, one obtains uniform bounds for embedded cycles and nested embedded cycles. The root \(p\) is used only to decide which pairs are counted as nested; it is not needed for finiteness.
\end{remark}

\section{Unboundedly many nested loops in the smooth category}\label{sec:counter-example}

\begin{proof}[Proof of Theorem~\ref{thm:smooth-counterexample}]
    We use the construction of B\'erard--Charron--Helffer \cite[Proposition 5.1]{MR4467996}. In spherical coordinates \((\theta,\varphi)\), they construct a smooth function
    \begin{align}
        \Phi(\theta,\varphi) = T(\theta)\cos\varphi
    \end{align}
    and a smooth positive conformal factor \(G\) such that, for the metric \(g_G=Gg_0\) on \(S^2\),
    \begin{align}
        -\Delta_{g_G}\Phi=2\Phi.
    \end{align}
    We recall only the part of their construction needed here; see Section~5 of \cite{MR4467996} for the full construction. On an equatorial interval \(\theta\sim \pi/2\), the function \(T\) is written as
    \begin{align}
        T(\theta)=V(\theta)+u_{\mathrm{osc}}(\theta),
    \end{align}
    where \(V\equiv 1\) on the support of the oscillatory term \(u_{\mathrm{osc}}\), and \(u_{\mathrm{osc}}\) is a nonzero constant multiple of a rescaled copy of
    \begin{align}
        v(t)= \exp\!\left(\frac{1}{(t+1)(t-1)}\right) \cos\!\left(\frac{1}{(t+1)(t-1)}\right), \qquad -1<t<1,
    \end{align}
    extended by zero outside \((-1,1)\). Consequently, on this support, \(T-1\) is a nonzero constant multiple of a rescaling of \(v\). The function \(v\) has infinitely many sign-changing zeros accumulating at \(t=\pm1\). Therefore \(T-1\) has infinitely many simple zeros in the equatorial band, and its sign alternates between successive simple zeros.
    
    We now construct infinitely many double nests. Choose infinitely many pairwise disjoint open intervals \(I_j=(a_j,b_j)\) between consecutive simple zeros such that
    \begin{align*}
        T(\theta)>1 \quad\text{for }\theta\in I_j, \qquad T(a_j)=T(b_j)=1.
    \end{align*}
     We choose them so that the closed intervals \([a_j,b_j]\) are pairwise disjoint. For each \(j\), define
    \begin{align}\label{eq:uj}
        U_j = \left\{ (\theta,\varphi): \theta\in I_j,\  |\varphi|<\arccos\frac{1}{T(\theta)} \right\},
    \end{align}
    where \(\varphi\) is taken in the coordinate interval \((-\pi,\pi]\). Since \(T>1\) on \(I_j\), the set \(U_j\) is nonempty. Since \(T\ge0\) and \(T\le1\) away from the intervals on which \(T>1\), the inequality \(T(\theta)\cos\varphi>1\) can hold only when \(T(\theta)>1\). Hence the sets \(U_j\) are precisely the connected components of the superlevel set \(\{\Phi>1\}\) arising from these positive intervals. In particular, each \(U_j\) is a topological disk and \(\gamma_j\coloneqq \partial U_j\) is contained in \(\{\Phi=1\}\).
    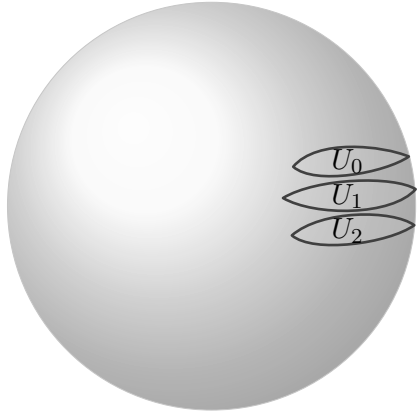
\begin{figure}[!htb]
        \centering
        \begin{tikzpicture}[scale=0.9,line cap=round,line join=round]
            \shade[ball color=gray!10,opacity=0.55] (0,0) circle (3);
            \draw[gray!45,line width=0.35pt] (0,0) circle (3);
            
            \filldraw[fill=gray!35,fill opacity=0.30,draw=black!75,line width=1.0pt]
            (1.20,0.58).. controls (1.55,1) and (2.52,0.89) .. (2.90,0.73)
              .. controls (2.55,0.49) and (1.52,0.3) .. (1.20,0.58)--cycle;
            \node at (2,0.66) {\(U_0\)};
    
            \filldraw[fill=gray!35,fill opacity=0.30,draw=black!75,line width=1.0pt]
            (1.05,0.12).. controls (1.55,0.42) and (2.68,0.45) .. (3,0.25)
              .. controls (2.65,-0.15) and (1.52,-0.15) .. (1.05,0.12)--cycle;
            \node at (2,0.15) {\(U_1\)};
            
            \filldraw[fill=gray!35,fill opacity=0.30,draw=black!75,line width=1.0pt]
            (1.18,-0.43).. controls (1.62,-0.05) and (2.66,-0.05) .. (2.98,-0.28)
              .. controls (2.60,-0.55) and (1.50,-0.7) .. (1.18,-0.43)--cycle;
            \node at (2,-0.36) {\(U_2\)};
        \end{tikzpicture}
        \caption{A schematic of the disks \(U_0,U_1,U_2\) defined in \eqref{eq:uj} on \(S^2\).}
        \label{fig:counter-example-on-sphere}
    \end{figure}
    Moreover, \(\gamma_j\) is a smooth embedded circle. Away from the two points \((a_j,0)\) and \((b_j,0)\), this follows from
    \begin{align}
        \partial_\varphi(T(\theta)\cos\varphi-1) = -T(\theta)\sin\varphi \neq0.
    \end{align}
    At the endpoints, we instead have
    \begin{align}
        \partial_\theta(T(\theta)\cos\varphi-1) = T'(\theta)\cos\varphi = T'(\theta),
    \end{align}
    which is nonzero because \(a_j\) and \(b_j\) are simple zeros of \(T-1\). Thus the level set is regular along \(\gamma_j\). Since the equation \(T(\theta)\cos\varphi=1\) cannot hold where \(T<1\), no level-set arc connects \(\gamma_j\) to another such oval through the intervening bands. Therefore each \(\gamma_j\) is a connected component of \(\{\Phi=1\}\). We have obtained infinitely many pairwise disjoint smooth ovals \(\gamma_j\subset\{\Phi=1\}\).
    
    Now fix the prescribed root \(p\in S^2\). Choose a point \(q\in U_0\), and let \(h:S^2\to S^2\) be a diffeomorphism such that \(h(p)=q\). Pull back the metric and the eigenfunction by setting \(\widetilde g=h^*g_G\) and \(\widetilde\Phi=\Phi\circ h\). By naturality of the Laplacian, \(-\Delta_{\widetilde g}\widetilde\Phi=2\widetilde\Phi\). The constant function \(1\) is a \(0\)-eigenfunction of \((S^2,\widetilde g)\). Therefore \(u\coloneqq \widetilde\Phi-1\) is a linear combination of eigenfunctions with eigenvalues \(2\) and \(0\).

    Define
    \begin{align}
        \widetilde U_j=h^{-1}(U_j), \qquad \widetilde\gamma_j=h^{-1}(\gamma_j).
    \end{align}
    Then each \(\widetilde\gamma_j\) is a connected component of \(u^{-1}(0)\), and \(\partial\widetilde U_j=\widetilde\gamma_j.\) Moreover, \(p\in\widetilde U_0\). Hence the \(p\)-inside of \(\widetilde\gamma_0\) is the other disk bounded by \(\widetilde\gamma_0\):
    \begin{align}
        D_p(\widetilde\gamma_0) = S^2\setminus \overline{\widetilde U_0}.
    \end{align}
    For every \(j\ge1\), the point \(p\) does not lie in \(\overline{\widetilde U_j}\), since the closures of the \(\widetilde U_j\)'s are pairwise disjoint and \(p\in\widetilde U_0\). Thus
    \begin{align}
        D_p(\widetilde\gamma_j)=\widetilde U_j, \qquad j\ge1.
    \end{align}
    Therefore, for every \(j\ge1\),
    \begin{align}
        \overline{D_p(\widetilde\gamma_j)} = \overline{\widetilde U_j} \subset S^2\setminus\overline{\widetilde U_0} = D_p(\widetilde\gamma_0).
    \end{align}
    Hence each pair \((\widetilde\gamma_j,\widetilde\gamma_0),j\ge1,\) is a \(p\)-rooted double nest. Since there are infinitely many such indices \(j\), the function \(u=\widetilde\Phi-1\) has infinitely many \(p\)-rooted double nests.
\end{proof}

\begin{remark}
    \begin{enumerate}
        \item The counterexample uses the constant eigenfunction. Thus Theorem~\ref{thm:smooth-counterexample} does not address the stricter variant in which both eigenfunctions are required to have positive eigenvalue.

        \item In the counterexample, \(\widetilde\gamma_0\) is common to all double nests. It remains open whether one can arrange infinitely many double nests with no oval appearing in two different nests. For instance, \begin{tikzpicture}[baseline=-0.5ex,scale=0.7]
            \draw (0,0) circle (0.25) circle (0.10) (1,0) circle (0.25) circle (0.10);
        \end{tikzpicture} would count.
    \end{enumerate}
\end{remark}

\subsection{A question about depth}
The construction above produces infinitely many rooted double nests, but it does not produce arbitrarily long chains of nesting. Here the \emph{depth} of a chain is the number of loops in it; for instance,
\begin{tikzpicture}[baseline=-0.5ex,scale=0.7]
    \draw (0,0) circle (0.25);
    \draw (0,0) circle (0.15);
    \draw (0,0) circle (0.09);
\end{tikzpicture} has depth \(3\).

\begin{definition}
    For \(p\)-ovals \(\gamma,\eta\), write \(\gamma\prec_p \eta\) if \(\overline{D_p(\gamma)}\subset D_p(\eta).\)
    The \textit{\(p\)-rooted depth of \(u^{-1}(0)\)} is the supremum of the lengths of chains \(\gamma_1\prec_p \gamma_2\prec_p\cdots\prec_p\gamma_m\) of \(p\)-ovals contained in \(u^{-1}(0)\).
\end{definition}

The next lemma gives a basic finite-depth result for shifted eigenfunctions. The bound is in terms of the numbers of positive and negative nodal domains of the eigenfunction, and hence is finite by Courant's theorem.

\begin{lemma}\label{lem:shifted-eigenfunction-finite-depth}
    Let \((M^2,g)\) be a closed connected smooth Riemannian surface, and let \(\psi\) be a nonconstant eigenfunction,
    \begin{align*}
        -\Delta_g\psi=\lambda\psi,\qquad \lambda>0.
    \end{align*}
    Then, for every \(c\in\mathbb R\) and every \(p\in M\), the level set \(\{\psi=c\}\), equivalently the nodal set of \(\psi-c\), has finite \(p\)-rooted depth.
\end{lemma}

\begin{proof}
    Let \(N_+(\psi)\) and \(N_-(\psi)\) denote the numbers of positive and negative nodal domains of \(\psi\). These numbers are finite by Courant's nodal domain theorem.
    
    Let \(\gamma_1\prec_p\gamma_2\prec_p\cdots\prec_p\gamma_m\) be a chain of \(p\)-ovals contained in \(\{\psi=c\}\). Set \(D_i\coloneqq D_p(\gamma_i).\) Thus
    \begin{align*}
        \overline{D_1}\subset D_2,\quad
        \overline{D_2}\subset D_3,\quad \ldots,\quad
        \overline{D_{m-1}}\subset D_m.
    \end{align*}
    Define the layers
    \begin{align*}
        A_1\coloneqq D_1,
        \qquad
        A_i\coloneqq D_i\setminus \overline{D_{i-1}},
        \quad 2\le i\le m.
    \end{align*}
    Each \(A_i\) is connected, and \(\partial A_i\subset \{\psi=c\}.\)
    
    First suppose \(c=0\). Since \(\psi\) cannot vanish identically on any open set, each \(A_i\) contains a point where \(\psi\ne0\). The nodal domain of \(\psi\) containing that point is contained in \(A_i\), because \(\partial A_i\subset\{\psi=0\}\). Distinct layers give distinct nodal domains. Hence
    \begin{align*}
        m\le N_+(\psi)+N_-(\psi).
    \end{align*}
    
    Now suppose \(c>0\). We claim that for each \(1\le i<m\), at least one of the two adjacent layers \(A_i,A_{i+1}\) contains a negative point of \(\psi\). Indeed, let \(A\) be one of the layers. If \(\psi\ge0\) in \(A\), then in fact \(\psi\ge c\) in \(A\). To see this, the strong maximum principle gives \(\psi>0\) in \(A\). If \(\psi<c\) somewhere in \(A\), then \(\psi\) attains a positive interior minimum, since \(\psi=c\) on \(\partial A\). At such a point, \(\Delta_g\psi\ge0,\) whereas the eigenfunction equation gives
    \begin{align*}
        \Delta_g\psi=-\lambda\psi<0,
    \end{align*}
    a contradiction. Thus \(\psi\ge c\) in \(A\).
    
    If neither \(A_i\) nor \(A_{i+1}\) contained a negative point, then the preceding paragraph would give
    \begin{align*}
        \psi\ge c
        \qquad\text{on } A_i\cup A_{i+1}.
    \end{align*}
    Since \(\psi=c\) on the intermediate oval \(\gamma_i\), every point of \(\gamma_i\) would be a local minimum point of \(\psi\) with value \(c>0\). This is impossible, because at such a point
    \begin{align*}
        \Delta_g\psi\ge0,
        \qquad\text{but}\qquad
        \Delta_g\psi=-\lambda c<0.
    \end{align*}
    The claim follows.
    
    Pair the layers as \((A_1,A_2),\ (A_3,A_4),\ldots.\) For each pair, the claim gives a layer containing a negative point. The negative nodal domain containing that point is contained in the chosen layer, because the boundary of every layer lies in \(\{\psi=c\}\) and \(c>0\). Since the chosen layers are pairwise disjoint, the corresponding negative nodal domains are distinct. Therefore
    \begin{align*}
        \left\lfloor \frac m2\right\rfloor\le N_-(\psi),
    \end{align*}
    and hence \(m\le 2N_-(\psi)+1.\)
    
    The case \(c<0\) is identical after reversing signs. Equivalently, apply the previous argument to \(-\psi\) and \(-c>0\). This gives \(m\le 2N_+(\psi)+1.\)
    
    In all cases the length \(m\) of the chain is bounded by a finite quantity. Thus \(\{\psi=c\}\) has finite \(p\)-rooted depth.
\end{proof}

\begin{remark}
    \begin{enumerate}
        \item It remains open whether one can prove finite depth for nodal sets of linear combinations of nonconstant eigenfunctions. One obstruction is the possible structure of connected components of common nodal sets.

        \item In a companion paper \cite{georgiev-serrano-koirala-logunov-inprep}, for any given \(N\), we construct a biharmonic polynomial with nested nodal loops of depth \(N\). This does not contradict Theorem~\ref{thm:plane-growth} because the latter is a statement for a fixed growth order, whereas the degrees, and hence the growth orders, of the polynomials in the companion construction are unbounded as \(N\to\infty\). 
    \end{enumerate}
        
\end{remark}


\bibliographystyle{amsalpha}
\bibliography{references}

@article {MR3890963,
    AUTHOR = {B\'erard, Pierre and Helffer, Bernard},
     TITLE = {On {C}ourant's nodal domain property for linear combinations
              of eigenfunctions. {P}art 1},
   JOURNAL = {Doc. Math.},
  FJOURNAL = {Documenta Mathematica},
    VOLUME = {23},
      YEAR = {2018},
     PAGES = {1561--1585},
      ISSN = {1431-0635,1431-0643},
   MRCLASS = {35P05 (35J05)},
  MRNUMBER = {3890963},
MRREVIEWER = {Dario\ Mazzoleni},
}

@article {MR4060310,
    AUTHOR = {B\'erard, Pierre and Helffer, Bernard},
     TITLE = {Sturm's theorem on the zeros of sums of eigenfunctions:
              {G}elfand's strategy implemented},
   JOURNAL = {Mosc. Math. J.},
  FJOURNAL = {Moscow Mathematical Journal},
    VOLUME = {20},
      YEAR = {2020},
    NUMBER = {1},
     PAGES = {1--25},
      ISSN = {1609-3321,1609-4514},
   MRCLASS = {34C10 (34L05 58J50)},
  MRNUMBER = {4060310},
MRREVIEWER = {Namig\ J.\ Guliyev},
       DOI = {10.17323/1609-4514-2020-20-1-1-25},
       URL = {https://doi.org/10.17323/1609-4514-2020-20-1-1-25},
}

@article {MR4082305,
    AUTHOR = {B\'erard, Pierre and Helffer, Bernard},
     TITLE = {Sturm's theorem on zeros of linear combinations of
              eigenfunctions},
   JOURNAL = {Expo. Math.},
  FJOURNAL = {Expositiones Mathematicae},
    VOLUME = {38},
      YEAR = {2020},
    NUMBER = {1},
     PAGES = {27--50},
      ISSN = {0723-0869,1878-0792},
   MRCLASS = {34-03 (01A55 34B24 34C10 34L10)},
  MRNUMBER = {4082305},
       DOI = {10.1016/j.exmath.2018.10.002},
       URL = {https://doi.org/10.1016/j.exmath.2018.10.002},
}

@incollection {MR4281621,
    AUTHOR = {B\'erard, Pierre and Helffer, Bernard},
     TITLE = {On {C}ourant's nodal domain property for linear combinations
              of eigenfunctions part {II}},
 BOOKTITLE = {Schr\"odinger operators, spectral analysis and number theory},
    SERIES = {Springer Proc. Math. Stat.},
    VOLUME = {348},
     PAGES = {47--88},
 PUBLISHER = {Springer, Cham},
      YEAR = {[2021] \copyright 2021},
      ISBN = {978-3-030-68489-1; 978-3-030-68490-7},
   MRCLASS = {35P05 (58J50)},
  MRNUMBER = {4281621},
       DOI = {10.1007/978-3-030-68490-7\_4},
       URL = {https://doi.org/10.1007/978-3-030-68490-7_4},
}

@article {MR4467996,
    AUTHOR = {B\'erard, Pierre and Charron, Philippe and Helffer, Bernard},
     TITLE = {Non-boundedness of the number of super level domains of
              eigenfunctions},
   JOURNAL = {J. Anal. Math.},
  FJOURNAL = {Journal d'Analyse Math\'ematique},
    VOLUME = {146},
      YEAR = {2022},
    NUMBER = {1},
     PAGES = {127--164},
      ISSN = {0021-7670,1565-8538},
   MRCLASS = {58J50 (35P15)},
  MRNUMBER = {4467996},
MRREVIEWER = {Yawei\ Chu},
       DOI = {10.1007/s11854-021-0189-9},
       URL = {https://doi.org/10.1007/s11854-021-0189-9},
}

@article{Boggio1901,
  author  = {Boggio, T.},
  title   = {Sull'equilibrio delle piastre elastiche incastrate},
  journal = {Rend. Acc. Lincei},
  volume  = {10},
  year    = {1901},
  pages   = {197--205}
}

@article{Boggio1905,
  author  = {Boggio, T.},
  title   = {Sulle funzioni di {Green} d'ordine \(m\)},
  journal = {Rend. Circ. Mat. Palermo},
  volume  = {20},
  year    = {1905},
  pages   = {97--135}
}

@article {MR4190398,
    AUTHOR = {Buhovsky, Lev and Logunov, Alexander and Sodin, Mikhail},
     TITLE = {Eigenfunctions with infinitely many isolated critical points},
   JOURNAL = {Int. Math. Res. Not. IMRN},
  FJOURNAL = {International Mathematics Research Notices. IMRN},
      YEAR = {2020},
    NUMBER = {24},
     PAGES = {10100--10113},
      ISSN = {1073-7928,1687-0247},
   MRCLASS = {58J50},
  MRNUMBER = {4190398},
MRREVIEWER = {Akira\ Asada},
       DOI = {10.1093/imrn/rnz181},
       URL = {https://doi.org/10.1093/imrn/rnz181},
}

@article {MR4704545,
    AUTHOR = {Chanillo, Sagun and Logunov, Alexander and Malinnikova,
              Eugenia and Mangoubi, Dan},
     TITLE = {Local version of {C}ourant's nodal domain theorem},
   JOURNAL = {J. Differential Geom.},
  FJOURNAL = {Journal of Differential Geometry},
    VOLUME = {126},
      YEAR = {2024},
    NUMBER = {1},
     PAGES = {49--63},
      ISSN = {0022-040X,1945-743X},
   MRCLASS = {58J50 (31C12 35J15 35P10 35R01)},
  MRNUMBER = {4704545},
MRREVIEWER = {Emmanuel\ Schenck},
       DOI = {10.4310/jdg/1707767334},
       URL = {https://doi.org/10.4310/jdg/1707767334},
}

@article {MR397805,
    AUTHOR = {Cheng, Shiu Yuen},
     TITLE = {Eigenfunctions and nodal sets},
   JOURNAL = {Comment. Math. Helv.},
  FJOURNAL = {Commentarii Mathematici Helvetici},
    VOLUME = {51},
      YEAR = {1976},
    NUMBER = {1},
     PAGES = {43--55},
      ISSN = {0010-2571,1420-8946},
   MRCLASS = {58G99 (35P15)},
  MRNUMBER = {397805},
MRREVIEWER = {Sh\^ukichi\ Tanno},
       DOI = {10.1007/BF02568142},
       URL = {https://doi.org/10.1007/BF02568142},
}

@book {MR65391,
    AUTHOR = {Courant, R. and Hilbert, D.},
     TITLE = {Methods of mathematical physics. {V}ol. {I}},
 PUBLISHER = {Interscience Publishers, Inc., New York},
      YEAR = {1953},
     PAGES = {xv+561},
   MRCLASS = {79.0X},
  MRNUMBER = {65391},
MRREVIEWER = {J.\ B.\ Diaz},
}

@article {MR1205487,
    AUTHOR = {Donnelly, Harold},
     TITLE = {Nodal sets for sums of eigenfunctions on {R}iemannian
              manifolds},
   JOURNAL = {Proc. Amer. Math. Soc.},
  FJOURNAL = {Proceedings of the American Mathematical Society},
    VOLUME = {121},
      YEAR = {1994},
    NUMBER = {3},
     PAGES = {967--973},
      ISSN = {0002-9939,1088-6826},
   MRCLASS = {58G25},
  MRNUMBER = {1205487},
MRREVIEWER = {Stig\ I.\ Andersson},
       DOI = {10.2307/2160300},
       URL = {https://doi.org/10.2307/2160300},
}

@article {MR29021,
    AUTHOR = {Duffin, R. J.},
     TITLE = {On a question of {H}adamard concerning super-biharmonic
              functions},
   JOURNAL = {J. Math. Physics},
  FJOURNAL = {Journal of Mathematics and Physics},
    VOLUME = {27},
      YEAR = {1949},
     PAGES = {253--258},
      ISSN = {0097-1421},
   MRCLASS = {31.0X},
  MRNUMBER = {29021},
MRREVIEWER = {M.\ Brelot},
}

@article {MR46440,
    AUTHOR = {Garabedian, P. R.},
     TITLE = {A partial differential equation arising in conformal mapping},
   JOURNAL = {Pacific J. Math.},
  FJOURNAL = {Pacific Journal of Mathematics},
    VOLUME = {1},
      YEAR = {1951},
     PAGES = {485--524},
      ISSN = {0030-8730,1945-5844},
   MRCLASS = {30.0X},
  MRNUMBER = {46440},
MRREVIEWER = {Z.\ Nehari},
       URL = {http://projecteuclid.org/euclid.pjm/1103052020},
}

@unpublished{georgiev-serrano-koirala-logunov-inprep,
  author = {Javier Gómez-Serrano and Robert Koirala and Alexander Logunov},
  title = {Nested nodal loops of biharmonic functions},
  note = {preprint},
  year = {2026}
}

@article {MR2591975,
    AUTHOR = {Grunau, Hans-Christoph and Robert, Fr\'ed\'eric},
     TITLE = {Positivity and almost positivity of biharmonic {G}reen's
              functions under {D}irichlet boundary conditions},
   JOURNAL = {Arch. Ration. Mech. Anal.},
  FJOURNAL = {Archive for Rational Mechanics and Analysis},
    VOLUME = {195},
      YEAR = {2010},
    NUMBER = {3},
     PAGES = {865--898},
      ISSN = {0003-9527,1432-0673},
   MRCLASS = {35J40 (31B30 35A08)},
  MRNUMBER = {2591975},
MRREVIEWER = {Ognyan\ Kounchev},
       DOI = {10.1007/s00205-009-0230-0},
       URL = {https://doi.org/10.1007/s00205-009-0230-0},
}

@incollection{Hadamard1908,
  author    = {Hadamard, J.},
  title     = {M{\'e}moire sur le probl{\`e}me d'analyse relatif {\`a}
               l'{\'e}quilibre des plaques {\'e}lastiques encastr{\'e}es},
  booktitle = {{\OE}uvres de {Jacques Hadamard}, Tome II},
  publisher = {CNRS},
  address   = {Paris},
  year      = {1968},
  pages     = {515--641},
  note      = {Original memoir, 1908}
}

@article {MR1696190,
    AUTHOR = {Hedenmalm, H{\aa}kan and Jakobsson, Stefan and Shimorin,
              Sergei},
     TITLE = {A maximum principle \`a{} la {H}adamard for biharmonic
              operators with applications to the {B}ergman spaces},
   JOURNAL = {C. R. Acad. Sci. Paris S\'er. I Math.},
  FJOURNAL = {Comptes Rendus de l'Acad\'emie des Sciences. S\'erie I.
              Math\'ematique},
    VOLUME = {328},
      YEAR = {1999},
    NUMBER = {11},
     PAGES = {973--978},
      ISSN = {0764-4442},
   MRCLASS = {31A30 (35B50 46E15)},
  MRNUMBER = {1696190},
MRREVIEWER = {Alain\ Brillard},
       DOI = {10.1016/S0764-4442(99)80308-5},
       URL = {https://doi.org/10.1016/S0764-4442(99)80308-5},
}

@article{hallgren2025structure,
  title={Structure Theory of Parabolic Nodal and Singular Sets},
  author={Hallgren, Max and Koirala, Robert and Ma, Zilu},
  journal={arXiv preprint arXiv:2511.11570},
  year={2025},
  pages={1-116}
}

@article {MR454051,
    AUTHOR = {Hardt, Robert M.},
     TITLE = {Triangulation of subanalytic sets and proper light subanalytic
              maps},
   JOURNAL = {Invent. Math.},
  FJOURNAL = {Inventiones Mathematicae},
    VOLUME = {38},
      YEAR = {1976/77},
    NUMBER = {3},
     PAGES = {207--217},
      ISSN = {0020-9910,1432-1297},
   MRCLASS = {32B20},
  MRNUMBER = {454051},
MRREVIEWER = {S.\ \L ojasiewicz},
       DOI = {10.1007/BF01403128},
       URL = {https://doi.org/10.1007/BF01403128},
}

@article {MR564475,
    AUTHOR = {Hardt, Robert M.},
     TITLE = {Semi-algebraic local-triviality in semi-algebraic mappings},
   JOURNAL = {Amer. J. Math.},
  FJOURNAL = {American Journal of Mathematics},
    VOLUME = {102},
      YEAR = {1980},
    NUMBER = {2},
     PAGES = {291--302},
      ISSN = {0002-9327,1080-6377},
   MRCLASS = {32B20 (32C05)},
  MRNUMBER = {564475},
MRREVIEWER = {Margherita\ Galbiati},
       DOI = {10.2307/2374240},
       URL = {https://doi.org/10.2307/2374240},
}

@article{lin-liu2016-betti-numbers-of-level-sets,
    AUTHOR = {Lin, Fanghua and Liu, Dan},
     TITLE = {On the {B}etti numbers of level sets of solutions to elliptic
              equations},
   JOURNAL = {Discrete Contin. Dyn. Syst.},
  FJOURNAL = {Discrete and Continuous Dynamical Systems. Series A},
    VOLUME = {36},
      YEAR = {2016},
    NUMBER = {8},
     PAGES = {4517--4529},
      ISSN = {1078-0947,1553-5231},
   MRCLASS = {35J25 (35B05 35B35 82D55)},
  MRNUMBER = {3479524},
MRREVIEWER = {Siegfried\ Carl},
       DOI = {10.3934/dcds.2016.36.4517},
       URL = {https://doi.org/10.3934/dcds.2016.36.4517},
}

@article {MR1267051,
    AUTHOR = {Shapiro, Harold S. and Tegmark, Max},
     TITLE = {An elementary proof that the biharmonic {G}reen function of an
              eccentric ellipse changes sign},
   JOURNAL = {SIAM Rev.},
  FJOURNAL = {SIAM Review. A Publication of the Society for Industrial and
              Applied Mathematics},
    VOLUME = {36},
      YEAR = {1994},
    NUMBER = {1},
     PAGES = {99--101},
      ISSN = {1095-7200},
   MRCLASS = {35J40},
  MRNUMBER = {1267051},
MRREVIEWER = {Yehuda\ Pinchover},
       DOI = {10.1137/1036005},
       URL = {https://doi.org/10.1137/1036005},
}

\end{document}